\documentclass{amsart}
\usepackage{mathptmx, amscd, amssymb, colonequals, url, tikz, bm, enumerate, verbatim, tikz-cd}
\makeatletter
\@namedef{subjclassname@2010}{%
  \textup{2010} Mathematics Subject Classification}
\makeatother
\usepackage[colorlinks,citecolor=red,urlcolor=blue,bookmarks=false,hypertexnames=true]{hyperref}

\newtheorem{theorem}{Theorem}[section]
\newtheorem{lemma}[theorem]{Lemma}

\newtheorem{proposition}[theorem]{Proposition}

\theoremstyle{definition}
\newtheorem{definition}[theorem]{Definition}
\newtheorem{example}[theorem]{Example}
\newtheorem{remark}[theorem]{Remark}
\numberwithin{equation}{theorem}

\def\mapsto{\longmapsto}

\newcommand{\RN}[1]{%
  \textup{\uppercase\expandafter{\romannumeral#1}}%
}

\begin{document}
\title [Cohomological dimension of ideals defining Veronese subrings]{Cohomological dimension of ideals \\ defining Veronese subrings}

\author[Vaibhav Pandey]{Vaibhav Pandey}
\date{\today}
\address{Department of Mathematics, University of Utah, 155 S 1400 E, Salt Lake City,\newline UT~84112, USA}
\email{pandey@math.utah.edu}



\begin{abstract}
Given a standard graded polynomial ring over a commutative Noetherian ring $A$, we prove that the cohomological dimension and the height of the ideals defining any of its Veronese subrings are equal. This result is due to Ogus when $A$ is a field of characteristic zero, and follows from a result of Peskine and Szpiro when $A$ is a field of positive characteristic; our result applies, for example, when $A$ is the ring of integers.
\end{abstract}
\maketitle

\section{Introduction}

Throughout this paper, all rings are assumed to be commutative, Noetherian, and with an identity element.

Let $T = \mathbb{Z} [x_1,x_2,\ldots,x_k]$ be the standard graded polynomial ring in $n$ indeterminants over the integers. Consider a minimal minimal presentation of its $n$-th Veronese subring $T^{(n)} = \oplus _{i \geq 0} T_{in}$ as $T^{(n)} \cong \mathbb{Z} [t_1, \ldots, t_d]/I$. We say that $I$ is the ideal defining the $n$-th Veronese subring of $T$. For $A$ a ring, we set $T_A = T\otimes _{\mathbb{Z}} A$.

Ogus in \cite[Example 4.6]{ogus} proved that when $A$ is a field of characteristic zero, the cohomological dimension of $I$ is the same as its height. The same result also follows when $A$ is a field of positive characteristic by a result of Peskine and Szpiro \cite[Proposition \RN{3}.4.1]{PS}. We prove that this continues to hold for any commutative Noetherian ring $A$. The critical step is the calculation of local cohomology of the polynomial ring $\mathbb{Z}[t_1, \ldots t_d]$ supported at the ideal $I$. More precisely, we prove:

\begin{theorem} \label{main}
Let $T = \mathbb{Z}[x_1, \ldots, x_k]$  be a polynomial ring with the $\mathbb{N}$-grading $[T]_0 = \mathbb{Z}$ and deg $x_i = 1$ for each $i$. Consider a minimal presentation of $T^{(n)}$ as $R/I$. Then \[H^i_I(R) = 0  \quad \text{for } i \neq \mathrm{height}\ I. \]
\end{theorem}

Towards the above result, we establish a condition for the injectivity of multiplication by a prime integer on local cohomology modules over the ring $\mathbb{Z}[t_1, \ldots t_d]$ in Lemma~\ref{lc-inj}. This strengthens \cite[Corollary 2.18]{LSW} and is a result of independent interest.

It is worth mentioning that in the above context, the arithmetic rank may vary with the characteristic of the ring $A$:

\begin{example}
Let $k[x_1, \ldots, x_n]$ be a standard graded polynomial ring over a field $k$. Let $R$ be a polynomial ring over $k$ in indeterminants that map entrywise to the distinct elements of the matrix
\begin{center}
\ensuremath{\begin{pmatrix}
x_1^2 & x_1x_2 & \cdots & x_1x_n \\
x_1x_2 & x_2^2 & \cdots & x_2x_n \\
\vdots  & \vdots  & \ddots & \vdots  \\
x_1x_n & x_2x_n & \cdots & x_n^2
\end{pmatrix}} .
\end{center}
Thus, $R$ is a polynomial ring in $n+1 \choose 2$ indeterminants. The relations between the generators of $R$ under the above map are precisely those corresponding to the size two minors of this matrix. These relations define an ideal $I$ of $R$, with $R/I$ being a minimal presentation. Barile proved that the arithmetic rank of $I$, i.e., the minimum number of equations defining the affine variety $V(I)$ set-theoretically, is

\begin{center}

\ensuremath{
\displaystyle{\mathrm{ara} \ I} =
\begin{cases}
  {n \choose 2} &
  \text{if char } $k = 2$, \\ {n+1 \choose 2} - 2 & \text{otherwise.}
\end{cases}}

\end{center}

 More generally, Barile computed the arithmetic rank of the class of ideals generated by the size $t$ minors of a symmetric $n \times n$ matrix of indeterminants over a field in \cite[Theorems 3.1, 5.1]{Ba} and remarked: \textit{This seems to be the first class of ideals defined over $\mathbb{Z}$ for which, after specialization to a field $k$, the arithmetical rank depends on $k$.} This dependence of the arithmetic rank of $I$ on the characteristic of the field makes it interesting to investigate the local cohomology of polynomial rings over the integers such as those examined here.

\end{example}

\section{Injectivity of multiplication by a prime integer\\ on local cohomology modules}

The following lemma gives a criterion for integer torsion in local cohomology modules of a standard graded polynomial ring over the integers:

\begin{lemma} \cite[Corollary 2.18]{LSW}
Let $R = \mathbb{Z}[x_1, \ldots, x_n]$ be a polynomial ring with the $\mathbb{N}$-grading $[R]_0 = \mathbb{Z}$ and deg $x_i = 1$ for each $i$. Let $I$ be a homogeneous ideal, $p$ a prime integer, and $h$ a nonnegative integer. Suppose that the Frobenius action on \[[H^{n-h} _{(x_1, \ldots, x_n)} (R/(I+pR))]_0\] is nilpotent, and that the multiplication by $p$ map \[H^{h+1}_I (R)_{x_i} \overset{.p} \rightarrow H^{h+1}_I (R)_{x_i}\] is injective for each $i$. Then the multiplication by $p$ map on $H^{h+1}_I (R)$ is injective.
\end{lemma}

The proof of this lemma largely relies on the following theorem. For an overview of $\mathcal{D}$-modules and $\mathcal{F}$-modules, we refer the reader to \cite{LSW}.

\begin{theorem} \cite[Theorem 2.16]{LSW}
Let $R$ be a standard graded polynomial ring, where $[R]_0$ is a field of prime characteristic. Let $\mathbf{m}$ be the homogeneous maximal ideal of $R$, and $I$ an arbitrary homogeneous ideal. For each nonnegative integer $k$, the following are equivalent:
\begin{enumerate}
    \item Among the composition factors of the Eulerian $\mathcal{D}$-module $\xi (H^k _I (R))$, there is at least one composition factor with support $\{\mathbf{m}\}$.
    \item Among the composition factors of the graded $\mathcal{F}$-finite module $H^k_I(R)$, there is at least one composition factor with support $\{\mathbf{m}\}$.
    \item $H^k_I(R)$ has a graded $\mathcal{F}$-module homomorphic image with support $\{\mathbf{m}\}$.
    \item The natural Frobenius action on $[H^{dimR - k}_{\mathbf{m}} (R/I)]_0$ is not nilpotent.
\end{enumerate}
\end{theorem}

We strengthen Lemma $2.1$ as follows:
\begin{lemma} \label{lc-inj}
Let $R = \mathbb{Z}[x_1, \ldots, x_n]$ be a polynomial ring with the $\mathbb{N}$-grading $[R]_0 = \mathbb{Z}$ and deg $x_i = 1$ for each $i$. Let $I$ be a homogeneous ideal, $p$ a prime integer, and $h$ a nonnegative integer. Let $t_1, \ldots, t_k$ be homogeneous elements in $R$ such that \[\sqrt{(t_1, \ldots, t_k)}R/I = (x_1, \ldots, x_n)R/I.\] Further, suppose that the Frobenius action on \[[H^{n-h} _{(t_1, \ldots, t_k)} (R/(I+pR))]_0\] is nilpotent and that the multiplication by $p$ map \[H^{h+1}_I (R)_{t_i} \overset{.p} \rightarrow H^{h+1}_I (R)_{t_i}\] is injective for each $i$. Then the multiplication by $p$ map on $H^{h+1}_I (R)$ is injective.
\end{lemma}

\begin{proof}
Since local cohomology modules depend only on the radical of the ideal defining the support, \[H^{n-h} _{(t_1, \ldots, t_k)} (R/(I+pR)) = H^{n-h} _{(x_1, \ldots, x_n)} (R/(I+pR)).\] Therefore, the natural Frobenius action on $[H^{n-h} _{(x_1, \ldots, x_n)} (R/(I+pR))]_0$ is nilpotent. The short exact sequence \[0 \rightarrow R \overset{.p} \rightarrow R \rightarrow R/pR \rightarrow 0\] induces the following long exact sequence of local cohomology modules: \[\cdots \rightarrow H^i_I(R) \rightarrow H^i_I(R/pR) \overset{\delta} \rightarrow H^{i+1} _I (R) \overset{.p} \rightarrow H^{i+1} _I (R) \rightarrow \cdots .\] Let $K$ denote the kernel of the multiplication by $p$ map in the above display, and $\mathbf{m}$ denote the homogeneous maximal ideal of $R/pR$.

By hypothesis, the localization $K_{t_i}$ is zero for each $i$. Thus, any prime ideal in the support of $K$ must contain each $t_i$. We may assume that $I$ is a proper ideal of $R$. Thus, prime ideals $\mathbf{p}$ in the support of $K$ are such that \[(t_1, \ldots, t_k)R \subseteq \mathbf{p} \text{  and  } I \subseteq \mathbf{p}.\] Therefore, $\sqrt{(t_1, \ldots, t_k)R + I} = \mathbf{m}$ is contained in $\mathbf{p}$. Thus, Supp($K$) is contained in $\{\mathbf{m}\}$.

The kernel $K$ is a $\mathcal{D}_{\mathbb{Z}}(R)$-module; since it is annihilated by $p$, it is also a module over \[\mathcal{D}_{\mathbb{Z}}(R)/p\mathcal{D}_{\mathbb{Z}}(R) \cong \mathcal{D}_{\mathbb{F}_p}(R/pR).\] This isomorphism follows, for example, from \cite[Lemma 2.1]{bblsz}. If $K$ is nonzero, then it is a homomorphic image of $H^i_I (R/pR)$ in the category of Eulerian graded $\mathcal{D}_{\mathbb{F}_p} (R/pR)$-modules, supported precisely at the homogeneous maximal ideal $\mathbf{m}$ of $R/pR$. But this is not possible, since the $\mathcal{D}_{\mathbb{F}_p} (R/pR)$-module $H^i_I (R/pR)$ has no composition factor with support $\{\mathbf{m}\}$ by Theorem $2.2$.
\end{proof}

We illustrate Lemma $2.3$ with the following example, but first a definition:

\begin{definition}
Let $I$ be an ideal of a ring $R$. For each $R$-module $M$, set \[\mathrm{cd}_R (I,M) = \sup \{n \in \mathbb{N} : H^n_I (M) \neq 0\}.\]

The \textit{cohomological dimension} of $I$ is \[\mathrm{cd}(I) = \sup \{\mathrm{cd}_R(I,M) : \text{$M$ is an $R$-module}\}.\]

By the right exactness of the functor $H^{\mathrm{cd}(I)} _I(-)$, we get $\mathrm{cd}_R (I) = \mathrm{cd}_R (I,R)$.
\end{definition}

\begin{example}
Consider the ring $T = \mathbb{Z}[x^4, x^3y,xy^3,y^4]$, which has a minimal presentation: \[T \cong \mathbb{Z}[t_1, t_2, t_3, t_4]/ (t_1t_4 - t_2t_3 \ , t_2t_4^2-t_3^3\ , t_1t_3^2-t_2^2t_4 \ , t_1^2t_3-t_2^3) = R/I.\]
We calculate the cohomological dimension of the ideal $I$. For any field $k$, we denote $T\otimes _{\mathbb{Z}} k$ by $T_k$. Hartshorne in \cite[Theorem]{hartshorne} showed that for $k$, a field of positive characteristic, the arithmetic rank of $IR_k$ is two. Since the ideal $I$ has height two, it follows that the cohomological dimension of $IR_k$ is also two.

We denote by $T'_k$ the ring $k[x^4, x^3y,x^2y^2,xy^3,y^4]$, which is the normalization of $T_k$ . The short exact sequence of $T_k$-modules \[0 \rightarrow T_k \rightarrow T'_k \rightarrow T'_k/T_k \rightarrow 0\] induces an isomorphism of local cohomology modules \[H^2 _{(x^4,x^3y,xy^3,y^4)} (T_k) \cong H^2 _{(x^4,x^3y,xy^3,y^4)} (T'_k),\] since $T'_k/T_k$ is a zero-dimensional $T_k$-module. As $T'_k$ is a direct summand of the polynomial ring $k[x,y]$, it follows that $[H^{2} _{(t_1, t_2,t_3,t_4)} (R/(I+pR))]_0 = 0$.

Note that $\sqrt{(t_1, t_4)}R/I = (x^4,x^3y,xy^3,y^4)R/I$. Further, \[IR_{t_1} = (t_3 - t_2^3/t_1^2 \ , t_4 - t_2^4/t_1^3 ) \text{ and } IR_{t_4} = (t_1 - t_3^4/t_4^3 \ , t_2 - t_3^3/t_4^2) \] are both two generated ideals. Thus, by Lemma~\ref{lc-inj}, the map $H^3_I(R) \overset{.p} \rightarrow H^3_I(R)$ is injective for each nonzero prime integer $p$. The exact sequence of local cohomology modules induced by \[0 \rightarrow R \overset{.p} \rightarrow R \rightarrow R/pR \rightarrow 0\] shows that $H^3_I(R) \overset{.p} \rightarrow H^3_I(R)$ is surjective since $H^3_I (R/pR) = 0$. Therefore, $H^3_I(R)$ is a $\mathbb{Q}$-vector space. But the cohomological dimension of $IR_{\mathbb{Q}}$ is known to be two. We conclude that the cohomological dimension of $I$ is two. It is worth noting that $T/pT \cong R/(I+pR)$ is not $F$-pure, since, \[(x^3y)^2 \notin (x^4)T/pT \text{ but } (x^3y)^{2p} \in (x^{4p})T/pT.\]
\end{example}

\section{Calculation of cohomological dimension}

\begin{definition}
Let $R = \oplus _{i \geq 0}  R_i$ be a graded ring, and $n$ be a positive integer. We denote by $R^{(n)}$, the \emph{Veronese subring} of $R$ spanned by elements which have degree a multiple of $n$, i.e., $R^{(n)} = \oplus _{i \geq 0} R_{in}$.
\end{definition}

We now present the key result which helps us calculate the cohomological dimension of ideals defining Veronese subrings.

\begin{proposition} \label{ci}
Let $A$ be a domain. Let $T = A[x_1, \ldots, x_k]$ be a polynomial ring with the $\mathbb{N}$-grading $[T]_0 = A$ and deg $x_i = 1$ for each $i$. Consider the lexicographic ordering of monomials in $T$ induced by $x_1 > x_2 > \cdots >x_k$.

Write a minimal presentation of $T^{(n)}$ as $R/I$ where $R = A[t_1, \ldots, t_d]$ with $t_i$ mapping to the $i$-th degree $n$ monomial under the above monomial ordering. Then, for each $i$ such that $t_i \mapsto x_j^n$ for some $j$, the ideal $IR_{t_i}$ is generated by a regular sequence of length  $\mathrm{height}\ I$.
\end{proposition}

\begin{proof}
By symmetry, it is enough to consider $t_1 \mapsto x_1^n$. We claim that the ideal $IR_{t_1}$ is generated by the regular sequence \[t_{k+1} - t_2 ^2/t_1,\ t_{k+2} - t_2t_3/t_1,\ t_{k+3} - t_2t_4/t_1,\ \ldots, t_{k+1 \choose 2} - t_k ^2/t_1,\ t_{{k+1 \choose 2}+1}-t_2 ^3/t_1^2, \ldots \]
\[\ldots,\ t_{k+2 \choose 3} - t_k^3/t_1^2\ , \ldots ,\ t_{d-1} - t_{k-1}t_2^{n-1}/t_1^{n-1},\ t_d - t_k^n/t_1^{n-1}.\]
Note that the length of this regular sequence is equal to $\mathrm{height}\ I$. Let $J$ be the ideal \[(t_{k+1} -\alpha _{k+1},\ t_{k+2} -\alpha_{k+2},\ \ldots,\ t_d - \alpha_d)R_{t_1},\] where $\alpha_{k+1},\alpha_{k+2}, \ldots \alpha_d$ are as above, i.e., $\alpha_{k+1} = t_2^2/t_1,\ \alpha_{k+2} = t_2t_3/t_1,\ \ldots,\text{ and } \alpha_d = t_k^n/t_1^{n-1}$. We claim that $J = IR_{t_1}$. It is clear that the ideal $J$ is contained in $IR_{t_1}$. Since $(R/I)_{t_1}$ is a subring of the fraction field of $R/I$, it follows that the ideal $IR_{t_1}$ is prime of height $d-k$.

Define a ring homomorphism $\phi \colon R_{t_1} \rightarrow A[t_1, \ldots, t_k][\frac{1}{t_1}]$ such that $t_i \mapsto t_i$ for $1 \leq i \leq k$ and $t_j \mapsto \alpha _j$ for $k+1 \leq j \leq d$. Then the map $\phi$ is a surjective ring homomorphism with kernel $J$. Hence, $J$ is a prime ideal of $R_{t_1}$ of height $d-k$. Thus, $J \subseteq IR_{t_1}$ are prime ideals of the same height in the ring $R_{t_1}$. We conclude that the ideals $J$ and $IR_{t_1}$ are equal.
\end{proof}

\begin{remark}
In the notation of Proposition~\ref{ci}, assume that the ring $A$ is regular. Then for each $t_i$ with $t_i \mapsto x_j^n$, the ring \[(R/I)_{t_i} \cong T_{x_j ^n} ^{(n)} = A[x_j ^n\ , 1/x_j ^n\ , x_2/x_j\ , \ldots, x_{j-1}/x_j\ , x_{j+1}/x_j\ , \ldots x_k/x_j]\] is regular.
\end{remark}

One of the most well-known vanishing results for local cohomology modules in positive characteristic was given by Peskine and Szpiro:
\begin{theorem} \cite[Proposition \RN{3}.4.1]{PS} \label{peskine}
Let $R$ be a regular domain of positive characteristic $p$ and $I$ be an ideal of $R$ such that $R/I$ is a Cohen-Macaulay ring. Then \[H^i_I(R) = 0  \quad \text{for } i \neq \mathrm{height}\ I. \]
\end{theorem}

The proof uses the flatness of the Frobenius action on $R$ which characterizes regular rings in positive characteristic.

Before we proceed to our main result, we would like to remark that the cohomological dimension of ideals may depend on the coefficient ring:

\begin{remark}
Let $k$ be a field. Let $R = \mathbb{Z}[u,v,w,x,y,z]$ and $R_k = R \otimes _{\mathbb{Z}}k$. Let $I$ be the ideal $(\Delta_1, \Delta_2, \Delta_3)R$ where $\Delta_1 = vz-wy$, $\Delta_2 = wx-uz$, and $\Delta_3 = uy-vx$. It is easily checked that $\mathrm{height}\ I = 2$. Then $\mathrm{cd}_{R/pR}(I, R/pR) =~2$ by Theorem~\ref{peskine}. However, Hochster observed that $H^3_I(R_{\mathbb{Q}})$ is nonzero, i.e., $\mathrm{cd}_{R_{\mathbb{Q}}}(I, R_{\mathbb{Q}}) =~3$. Since local cohomology commutes with localization, we also have $H^3_I(R)$ is nonzero, i.e., $\mathrm{cd}_{R}(I, R) = 3$. We point the reader to \cite[Example 21.31]{twentyfour} for further details.
\end{remark}

In Theorem~\ref{main}, we obtain a vanishing result for local cohomology modules over the integers similar to Theorem~\ref{peskine}.

\begin{proof} [Proof of Theorem \ref{main}]
Let $h$ denote the height of the ideal $I$. Since $R$ is regular, the grade of $I$ equals $h$ so that $H^i_I(R) = 0$ for $i <h$. Further, by Grothendieck's nonvanishing theorem, $H^h_I (R) \neq 0$.

Let $p$ be a nonzero prime integer. The short exact sequence \[0 \rightarrow R \overset{.p} \rightarrow R \rightarrow R/pR \rightarrow 0\] induces \[\cdots \rightarrow H^i_I(R) \rightarrow H^i_I(R/pR) \overset{\delta} \rightarrow H^{i+1} _I (R) \overset{.p} \rightarrow H^{i+1} _I (R) \rightarrow \cdots .\] Note that the height of the ideal $IR/pR$ is also $h$. Hence, by Theorem~\ref{peskine} \[H^i _I(R/pR) = H^i _{IR/pR}(R/pR) = 0 \text{\; for \;} i \neq h. \] It follows that the map $H^i_I(R)\overset{.p} \rightarrow H^i_I(R)$ is an isomorphism for each $i > h+1$ and that the map $H^{h+1}_I(R) \overset{.p} \rightarrow H^{h+1}_I(R)$ is surjective. The crucial part that remains to show is that the map $H^{h+1}_I(R) \overset{.p} \rightarrow H^{h+1}_I(R)$ is also injective. For this, we appeal to Lemma ~\ref{lc-inj}. After reordering of indices, let $t_1, \ldots, t_k$ denote the preimages of $x_1^n, \ldots, x_k^n$ respectively.

The ring $R/(I+pR)$ is a direct summand of the polynomial ring $T/pT$. Therefore, $[H^{n-h} _{(t_1, \ldots, t_k)} (R/(I+pR))]_0$ is zero.

By symmetry, it is enough to show that the multiplication by $p$ map \[H^{h+1}_I (R)_{t_1} \overset{.p} \rightarrow H^{h+1}_I (R)_{t_1}\] is injective. Note that the $R$-module $H^{h+1}_I (R)_{t_1}$ is isomorphic to $H^{h+1}_{IR_{t_1}}(R_{t_1})$. Applying Proposition ~\ref{ci} with $A = \mathbb{Z}$, we get that the ideal $IR_{t_1}$ is generated by a regular sequence of length $h$. Therefore, $H^{h+1}_{IR_{t_1}}(R_{t_1}) = 0$ and thus the map $H^{h+1}_I(R) \overset{.p} \rightarrow H^{h+1}_I(R)$ is injective.

For $i >h$, by \cite[Example 4.6]{ogus}, the module $H^i_I(R) \otimes_{\mathbb{Z}} \mathbb{Q}$ vanishes so that $H^i_I(R)$ is equal to its $\mathbb{Z}$-torsion submodule. But the $\mathbb{Z}$-torsion submodule of $H^i_I(R)$ is zero since multiplication by each nonzero prime integer is injective. We therefore conclude that the local cohomology modules $H^i_I(R)$ vanish for $i> h$.
\end{proof}

\begin{remark}
Following the notation of Theorem~\ref{main}, all but finitely many prime integers are known to be nonzerodivisors on $H^i_I(R)$ for any $i$ by \cite[Theorem 3.1 (2)]{bblsz}. Note that in Theorem~\ref{main}, we proved that \textit{each} nonzero prime integer is a nonzerodivisor on $H^i_I(R)$ for every $i$. Consequently, any associated prime of the $R$-module $H^h_I(R)$ contracts to the zero ideal in the integers.

In \cite[Section 4]{singh}, Singh constructs an example of a local cohomology module over a six dimensional hypersurface, which has $p$-torsion elements for \emph{each} prime integer $p$, and consequently has infinitely many associated prime ideals.
\end{remark}

In \cite[Theorem 4.1]{raicu}, Raicu recovers the result due to Ogus in \cite[Example 4.6]{ogus} which we used in proving Theorem \ref{main}; and also determines the $\mathcal{D}$-module structure of the only nonvanishing local cohomology module.

Finally, we extend Theorem~\ref{main} to standard graded polynomial rings with coefficients from any commutative Noetherian ring. For this, we use the following proposition which is proved in \cite{BV} more generally when $R = \mathbb{Z}[t_1, \ldots, t_d]/J$ is a faithfully flat $\mathbb{Z}$-algebra.

\begin{proposition} \cite[Proposition 3.14]{BV} \label{bv}
Let $I$ be an ideal of the polynomial ring $R = \mathbb{Z}[t_1, \ldots, t_d]$ and $A$ be a ring. If there exists an integer $h$ such that $\mathrm{grade}\ I(R\otimes _{\mathbb{Z}}k) = h$ for every field $k$, then $\mathrm{grade }\ I(R\otimes _{\mathbb{Z}}A) = h$. Analogous statements hold for height.
\end{proposition}

\begin{theorem}
Let $A$ be a commutative Noetherian ring and $T = A[x_1, \ldots, x_k]$  be a polynomial ring with the $\mathbb{N}$-grading $[T]_0 = A$ and deg $x_i = 1$ for each $i$. Consider a minimal presentation of $T^{(n)}$ as $R/I$. Then \[H^i_I(R) = 0  \quad \text{for } i \neq \mathrm{height}\ I. \]
\end{theorem}

\begin{proof}
Theorem ~\ref{main} and Proposition ~\ref{bv} together ensure that $\mathrm{height}$ $I$ and $\mathrm{grade}$ $I$ are equal. Therefore, $H^i _I(R) = 0 \text{\;  for \;} i < \text{height }I$. Further, the map $\mathbb{Z} \longrightarrow A$ induces the map $\mathbb{Z}[x_1, \ldots, x_n] \longrightarrow R$ which makes $R$ into a $\mathbb{Z}[x_1, \ldots, x_n]$-module. By the right exactness of the top local cohomology, the cohomological dimension of $I$ in $R$ is at most the cohomological dimension of $I$ in $\mathbb{Z}[x_1, \ldots, x_n]$, which, by Theorem~\ref{main}, equals height $I$.
\end{proof}

\section*{Acknowledgement}
The author would like to thank Anurag Singh for many valuable discussions and for his constant encouragement and support.

\bibliographystyle{alpha}
\bibliography{bibfile.bib}

\end{document}